\documentclass[a4paper,11pt]{article}
\usepackage{amssymb,amsthm,amsmath,latexsym}

\newtheorem{thm}{Theorem}[section]
\newtheorem{pro}[thm]{Proposition}
\newtheorem{lem}[thm]{Lemma}

\theoremstyle{definition}

\newtheorem{rem}[thm]{Remark}

\catcode`Ý=\active \defÝ{\. I} \catcode`ý=\active \defý{{\i }}
\catcode`ç=\active \defç{\c{c}} \catcode`Ü=\active \defÜ{\"{U}}

\date{}

\title{\normalsize\bf ON GROUPS WITH ALL SUBGROUPS SUBNORMAL OR SOLUBLE OF BOUNDED DERIVED LENGTH}

\author{\small{\textsc{Kývanç Ersoy\begin{footnote}{This study was carried out
during the visit of the first author to University of Salerno
which is supported by The Scientific and Technological Research
Council of Turkey (TÜBÝTAK) BÝDEB 2219 International Post Doctoral
Research Fellowship.  The first author thanks TÜBÝTAK
for the support.}\end{footnote}}}\\
\small{Department of Mathematics, Mimar Sinan Fine Arts University}\\
\small{Ýstanbul, 34427 - Turkey}\\
\small{E-mail: kivanc.ersoy@msgsu.edu.tr}\\
[10pt]
\small{\textsc{Antonio Tortora} and \textsc{Maria Tota}}\\
\small{Dipartimento di Matematica, Universit\`a di Salerno}\\
\small{Via Ponte don Melillo, 84084 - Fisciano (SA), Italy}\\
\small{E-mail: antortora@unisa.it, mtota@unisa.it}}

\begin{document}
\maketitle

\begin{abstract}
In this paper, we deal with locally graded groups whose subgroups are either subnormal or
 soluble of bounded derived length, say $d$. In particular,
we prove that every locally (soluble-by-finite) group with this property is either soluble
or an extension of a soluble group of derived length at most $d$
by a finite group, which fits between a minimal simple group and its automorphism group. We also classify all the finite non-abelian simple
groups whose proper subgroups are metabelian.\\

\end{abstract}

\section{Introduction}

A well-known result, due to W. M$\ddot{\rm{o}}$hres (see \cite{Mo2}), states that a group
with all subgroups subnormal is soluble, while a result proved, separately, by C. Casolo (see \cite{Ca2}) and
H. Smith (see \cite{Sm4}) shows that such a group is nilpotent if it is also torsion-free. Later,
Smith generalized these results to groups in which every subgroup is either subnormal or nilpotent. More precisely,
he proved, in \cite{Sm2}, that a locally (soluble-by-finite) group with all subgroups subnormal or
nilpotent is soluble, and the same holds for a locally graded group whose non-nilpotent
subgroups are subnormal of bounded defect. Also, in both cases, the nilpotence follows if
the group is torsion-free (see \cite{Sm3}). Recall that a group is {\em locally graded} if every non-trivial
finitely generated subgroup has a non-trivial finite image. The class of
locally graded groups is rather wide and, in particular, it contains all locally (soluble-by-finite) groups.
This restriction is made in order to avoid Tarski monsters
(see \cite{Ol}) which show that the previous results are false without any finiteness condition.

In this paper, we are interested in studying locally graded groups with all subgroups subnormal or soluble.
The first problem that arises here is the presence of finite {\em minimal simple groups}, i.e. non-abelian simple
groups in which every proper subgroup is soluble. They have been completely classified by J.\,G. Thompson in \cite{Th}.
Using this classification, in Section 2, we get all the finite non-abelian simple groups having
each proper subgroup metabelian.

Another difficulty is due to infinite locally graded groups with all proper subgroups soluble.
Such groups are both hyperabelian (see \cite{FdGN}) and locally soluble (see \cite{DES}),
but it is still an open question whether they are soluble. However, there is a positive answer
if we bound the derived length of subgroups (see \cite{DE}). Motivated by this result,
we deal with locally graded groups whose subgroups are either subnormal or soluble of bounded derived length.
In our analysis, {\em almost minimal simple groups} show up. These are groups which fit between
a minimal simple group and its automorphism group.

In line with Smith's results \cite{Sm3,Sm2}, our main theorems follow. They will be
proved in Section 3.

\begin{thm}\label{d}
Let $G$ be a locally (soluble-by-finite) group and suppose that, for some positive integer
$d$, every subgroup of $G$ is either subnormal or soluble of derived length at most $d$.
Then either
\begin{itemize}
\item[$(i)$] $G$ is soluble, or
\item[$(ii)$] $G^{(r)}$ is finite for some integer $r$ and $G$ is an extension of a
soluble group of derived length at most $d$ by a finite almost minimal simple group.
\end{itemize}
\end{thm}

\begin{thm}\label{dn}
Let $G$ be a locally graded group and suppose that, for some positive integers $n$ and $d$,
every subgroup of $G$ is either subnormal of defect at most $n$ or soluble of derived length
at most $d$. Then either
\begin{itemize}
\item[$(i)$] $G$ is soluble of derived length not exceeding a function depending on $n$ and
$d$, or
\item[$(ii)$] $G^{(r)}$ is finite for some integer $r=r(n)$ and $G$ is an extension of a
soluble group of derived length at most $d$ by a finite almost minimal simple group.
\end{itemize}
\end{thm}

\section{Minimal simple groups}

In this section we focus on locally graded minimal simple groups. By \cite[Lemma 2.4]{FdGN}
such groups are necessarily finite and they are known:

\begin{thm}[\cite{Th}, Corollary 1]\label{Th}
Every finite minimal simple group is isomorphic to one of the
following groups:
\begin{itemize}
\item[$(i)$] $PSL(2,2^p)$, where $p$ is any prime;
\item[$(ii)$] $PSL(2,3^p)$, where $p$ is any odd prime;
\item[$(iii)$] $PSL(2,p)$, where $p>3$ is any prime such that $p^{2}+1\equiv 0\; (mod\;5)$;
\item[$(iv)$] $PSL(3,3)$;
\item[$(v)$] $Sz(2^p)$, where $p$ is any odd prime.
\end{itemize}
\end{thm}

\vspace{0.3cm}

The table below, that will be useful later, shows some of the details concerning the outer
automorphism group of a finite minimal simple group $M$. By \cite[p. xv]{atlas}, $|Out\,(M)|=d\cdot f\cdot g$
where, $d$ is the order of the group of diagonal automorphisms,
$f$ is the order of the group of field automorphisms and $g$ is
the order of the group of graph automorphisms (modulo field automorphisms). For more details,
see \cite[Table 5, p. xvi]{atlas}.

\renewcommand{\baselinestretch}{1.2}\small

\begin{table}[h]
\centering
\begin{tabular}{|l|c|c|c|c|c|}

\hline

$M$ & $d$  & $f$ & $g$ & $|Out(M)|$ \\ \hline

$PSL(2,2^p)$ & $1$  & $p$ & $1$ & $p $ \\ \hline

$PSL(2,3^p)$, $p\geq 3$ & $2$  & $p$ & $1$ & $2p $ \\
\hline

$PSL(2,p)$, $p>3$ and $5|(p^{2}+1)$ & $2$ & $1$ & $1$ & $2$
\\ \hline

$ PSL(3,3)$ & $1$  & $1$ & $2$ & $2$ \\ \hline

$Sz(2^{p})$, $p\geq 3$ &  $1$ & $p$ & $1$ & $p$ \\ \hline

\end{tabular}
\caption{Outer automorphisms of a finite minimal
simple group\label{Table}  }
\end{table}

\renewcommand{\baselinestretch}{1}\normalsize
\vspace{0.2cm}

In light of Theorem \ref{Th}, we now classify all the finite non-abelian simple
groups whose proper subgroups are metabelian.

\begin{pro}\label{met}
Let $G$ be a finite non-abelian simple group with every proper subgroup metabelian.
Then $G$ is isomorphic to one of the following groups:
\begin{itemize}
\item[$(i)$] $PSL(2,2^p)$, where $p$ is any prime;
\item[$(ii)$] $PSL(2,3^p)$, where $p$ is any odd prime;
\item[$(iii)$] $PSL(2,p)$, where $p>3$ is any prime such that $p^{2}+1\equiv 0$ $(mod\;5)$
and $p^{2}-1\not\equiv 0$ $(mod\;16)$.
\end{itemize}
\end{pro}

\begin{proof}
It is enough to analyze each case of Theorem \ref{Th}.

Let $q$ be a power of any prime. By \cite[Theorem 6.25]{SuI},
$PSL(2,q)$ contains a non-metabelian soluble subgroup if and only
if it has a subgroup isomorphic to $S_4$, the symmetric group of degree 4. Also, by
\cite[Theorem 6.26]{SuI}, this is equivalent to the condition
$q^{2}\equiv 1$ $(mod\;16)$. Hence, if $q=2^{p}$, then all
subgroups of $PSL(2,2^p)$ are metabelian. Suppose $q=3^{p}$, with
$p=2k+1$. Since $9^{2k}\equiv 1$ $(mod\;16)$, $S_4$ is never
contained in $PSL(2,3^p)$ and therefore all subgroups of
$PSL(2,3^p)$ are metabelian. Let $q=p>3$ with $p^{2}+1\equiv 0\;
(mod\;5)$. If $p^{2}-1\not\equiv 0$ $(mod\;16)$, all subgroups of
$PSL(2,p)$ are metabelian.

Now, we have to consider $PSL(3,3)$ and $Sz(2^{p})$, $p\geq 3$. But $PSL(3,3)$ has
 a subgroup isomorphic to $SL(2,3)$, which has derived length $3$;
  so we finish with $Sz(q)$ where $q=2^p$. By \cite[Theorem
9]{Suz}, $Sz(q)$ contains a Frobenius group $F$ of order
$q^{2}(q-1)$. Moreover, $Sz(q)$ has only one abelian subgroup of
order dividing $q^{2}(q-1)$, that is cyclic of order $q-1$, and
its normalizer is a dihedral group of order $2(q-1)$ (see
\cite{Suz}, p. 137). Hence, $F$ is not metabelian.
\end{proof}

\begin{rem} We can observe that every proper subgroup of a
minimal simple group has derived length at most $5$.
By Theorem \ref{Th} and Proposition \ref{met}, we need to consider the following cases:\\

Let $G=PSL(2,p)$, $p>3$, $p^{2}+1\equiv 0$ $(mod\;5)$ and $p^{2}-1 \equiv 0$ $(mod\;16)$.
Then, by \cite[Theorems 6.25, 6.26]{SuI}, $G$ has a subgroup isomorphic to $S_{4}$, which is soluble of
derived length $3$. This is also the unique non-metabelian subgroup of $G$.\\

Let $G=PSL(3,3)$ and $H$ be a proper subgroup of $G$. Since $H$ is soluble, it contains a
non-trivial normal elementary abelian subgroup. Thus, by \cite[Theorem 7.1]{Blo}, one
of the following holds:

\begin{itemize}
\item[$(1)$] $H$ has a cyclic normal subgroup of index at most $3$;

\item[$(2)$] $H$ has an abelian normal subgroup $K$ such that $H/K$ can be embedded into
 the symmetric group $S_{3}$;

\item[$(3)$] $H$ has a normal elementary abelian $3$-subgroup $K$ such that $H/K$ can be embedded into
$GL(2,3)$. Now, the derived length of $GL(2,3)$ is $4$ and so $H$ has derived length at most $5$.
Indeed, let $Z=Z(SL(3,3))$ and $H$ be the subgroup of $G$ given by
\[\{\left(%
\begin{array}{ccc}
  a & b & c \\
  d & e & f \\
  0 & 0 & g \\
\end{array}%
\right)Z \;\;|\; a,b,c,d,e,f,g \in \mathbb{F}_{3}, \;\;
(ae-bd)g=1\}.\]
Then
\[K=\{\left(%
\begin{array}{ccc}
  1 & 0 & c \\
  0 & 1 & f \\
  0 & 0 & 1 \\
\end{array}%
\right)Z \;\;|\; c,f \in \mathbb{F}_{3} \}\]
is an elementary abelian $3$-subgroup of $H$ such that $H/K \cong GL(2,3)$.
\end{itemize}
Therefore, every proper subgroup of $PSL(3,3)$ has derived length
at most $5$ and $PSL(3,3)$ contains a subgroup of derived length
$5$.\\

Let $G=Sz(2^{p})$ for $p\geq 3$. Then, by \cite[Theorem 4.1]{wilson}, any maximal subgroup of
$G$ has derived length at most $3$.

\end{rem}

\section{Main results}

We start with some preliminary lemmas.

\begin{lem}\label{G^r}
Let $H$ be a subgroup of a group $G$. If every subgroup containing $H$ is subnormal in
 $G$, then $G^{(r)}\leq H$ for some $r\geq 0$. In particular, $r=0$ if and only if $H=G$.
\end{lem}

\begin{proof}
We may assume $H<G$. Then there exists a series from $H$ to $G$, and by \cite[Theorem 7]{Mo2},
each factor is soluble. Hence, we have an abelian series from $H$ to $G$,
say $H=H_0\lhd H_1\lhd\ldots\lhd H_{r}=G$. As $G^{(i)}\leq H_{r-i}$, for all $i\geq 0$,
we get $G^{(r)}\leq H$.
\end{proof}

\begin{rem}\label{rdependn} In the previous lemma, if we also assume that the subgroups
containing $H$ are subnormal in $G$ of defect at most $n\geq 1$, then by
\cite[12.2.8]{LR}, $r$ depends on $n$. See also \cite{Ca}.
\end{rem}

\begin{lem}\label{hyp}
Let $G$ be a locally graded group with all subgroups subnormal or soluble, and suppose that
$N$ is a minimal non-soluble normal subgroup of~$G$.
\begin{itemize}
\item[$(i)$] If $N$ is infinite, then $G$ is hyperabelian.
\item[$(ii)$] If $N$ is finite, then $G$ is an extension of a soluble group by a finite
 almost minimal simple group.
\end{itemize}
\end{lem}

\begin{proof}
First, notice that every subgroup of $G/N$ is subnormal, so that $G/N$ is soluble,
by \cite[Theorem 7]{Mo2}.\\

$(i)$ By \cite[Lemma 2.4]{FdGN}, $N$ is hyperabelian. Let
$K=\langle K_i: K_i<N$, $K_i\lhd G\rangle$. If $K=N$, since each
$K_i$ is soluble, $N$ has a $G$-invariant ascending abelian
series. Hence $G$ is hyperabelian, since $G/N$ is a soluble group.
Now, consider $K<N$ and assume for a contradiction that $G$ is not
hyperabelian. Then $K$ is soluble and so by \cite[Corollary]{LMS}
$G/K$ is locally graded. Moreover, $G/K$ is not
 hyperabelian and its normal subgroup $N/K$ is minimal non-soluble. However, $N/K$ is hyperabelian
  and thus $N/K$ is also infinite. We can therefore restrict to the case $K=1$.

Let $A$ be a non-trivial normal abelian subgroup of $N$. Then
$N=A^G$, so that $N$ is the product of normal abelian subgroups.
This implies that $N$ is locally nilpotent. If $T$ is its
 torsion subgroup, we have either $T=1$ or $T=N$. Let $T=1$. As solubility is a countably
 recognizable property, we have that $N$ is countable. It is also locally nilpotent and
 torsion-free. Then, by \cite[Lemma 2]{Mo}, there exists $M<N$ such that the isolator $I_N(M)$
 equals $N$. Also, $I_N(M)^{(i)}\leq I_N(M^{(i)})$ for all $i\geq 0$ (see, for instance,
 \cite[2.3.9]{LR}). As $M$ is soluble, so is $N$, a contradiction.
Assume $T=N$. Then $N$ is a locally finite $p$-group. Clearly $Z(N)=1$ and so we may apply
\cite[Lemma 2.1]{ASS} to $N$. It follows that there exists $m>0$, such that
$R=\langle Z(H):H\lhd N, d(H)>m\rangle$ is a proper subgroup of $N$, where $d(H)$ denotes
the derived length of $H$. On the other hand, $N$ has a finitely generated soluble subgroup of derived
length greater than $m$. This means that there is a subgroup $L$ of $N$ generated by finitely many
abelian normal subgroups that is necessarily nilpotent but of derived length $>m$. The set $J$ of
all such subgroups $L$ is invariant under $Aut\,(N)$. So the subgroup
$\bar{R}=\langle Z(L):L\in J\rangle$ is characteristic in $N$ and normal in $G$. Furthermore,
 $\bar{R}\leq R<N$. Thus $\bar{R}=1$ and this is a contradiction.\\

$(ii)$ Let $S$ be the soluble radical of $N$. Surely $S$ is a characteristic subgroup of
$N$ and so it is normal in $G$. Without loss of generality,
assume $S=1$. Then $N$ is a finite non-abelian simple group, hence
$C_{G}(N)\cap N$ is trivial. This gives that $C_{G}(N)$ is
soluble. Since $G/C_G(N)$ embeds in $Aut(N)$, we get that $G$ is
soluble-by-finite. This also implies that the soluble radical of
$G$ is soluble. Suppose it is trivial, so that $G$ is finite.

Let $M$ be a minimal normal subgroup of $G$. Then $M$ is the
direct product of copies of a non-abelian simple group $A$.
Obviously, $A$ is subnormal in $G$ and so, by Lemma \ref{G^r}, we
have $G^{(r)}\leq A$ for some $r>0$. Thus $A=M$. It follows that
$M$ is a minimal simple group and, by M$\ddot{\rm{o}}$hres' result
(see \cite{Mo2} or \cite[12.2.1]{LR}), $G/M$
is soluble. Moreover $C_{G}(M)\cap M=1$, so that $C_{G}(M)$ is a
normal soluble subgroup of $G$. By our assumption, $C_{G}(M)=1$
and this gives $G\lesssim Aut\,(M)$.
\end{proof}

\begin{lem}\label{SF}
Let $G$ be a locally (soluble-by-finite) group with all subgroups subnormal or soluble. Then either
\begin{itemize}
\item[$(i)$] $G$ is locally soluble, or
\item[$(ii)$] $G^{(r)}$ is finite for some integer $r$ and $G$ is an extension of a soluble
group by a finite almost minimal simple group.
\end{itemize}
\end{lem}

\begin{proof}
Let $K=G^{(\alpha)}=G^{(\alpha+1)}$, a
perfect group. If $K=1$, then $G$ has a descending normal series with abelian factors. Since $G$ is locally (soluble-by-finite),
it follows that $G$ is locally soluble. Suppose $K\neq 1$ and let $H$ be a proper subgroup of $K$ that is not
soluble. Then all subgroups containing $H$ are subnormal in $G$ and, by Lemma 3.1, there is an integer $r$ such that $G^{(r)}\leq H$,
a contradiction. Thus every proper subgroup of $K$ is soluble and $K$, being perfect, is
a minimal non-soluble group. Furthermore $G/K$ is soluble, so that $\alpha$ is
finite. The claim is now a consequence of Lemma \ref{hyp}.
\end{proof}

Notice that, proving $(ii)$ of Lemma \ref{SF}, we have that the
derived series of $G$ ends in finitely many steps. The next lemma,
which follows from \cite[Proposition 1]{Sm} together with
\cite[12.2.6]{LR}, shows that this also happens when $G$ is
locally soluble.

\begin{lem}\label{Sm}
Let $\mathfrak{X}=\bigcup_{i\in \mathbb{N}} \mathfrak{X}_i$ be a class of groups, where each class
$\mathfrak{X}_i$ is closed under taking subgroups and direct limits, and
 $\mathfrak{X}_i\subseteq \mathfrak{X}_{i+1}$ for all $i$. Let G be a group
 with all subgroups subnormal or in $\mathfrak{X}$, and suppose that $G\notin \mathfrak{X}$.
 If $G$ is locally soluble, then $G^{(r)}=G^{(r+1)}$ for some integer $r$.
\end{lem}

Now, we can prove Theorem \ref{d}.

\begin{proof}[{\bf Proof of Theorem \ref{d}}]
By Lemma \ref{SF}, $G$ is either locally soluble, or $G^{(r)}$ is
finite for some integer $r$ and $G$ is an extension of a soluble
group $S$ by a finite almost minimal simple group. If $S$ is not soluble of
derived length at most $d$ then, by Lemma \ref{G^r}, $G^{(r)}\leq S$ and
$G$ is soluble. Let $G$ be
locally soluble and suppose that it is not soluble. By Lemma
\ref{Sm} with $\mathfrak{X}$ the class of soluble groups, we have
$G^{(s)}=G^{(s+1)}$ for some $s\geq0$. Moreover, $G^{(s)}$ is not
soluble. It follows, as in the proof of Lemma \ref{SF}, that every
proper subgroup of $G^{(s)}$ is soluble of length at most $d$.
Thus $G^{(s)}$ is finite by \cite[Lemma 2.1]{DE}, a contradiction.
\end{proof}

By a theorem of J.\,E. Roseblade (see \cite{Ros} or
\cite[12.2.3]{LR}), a group in which every subgroup is subnormal
of defect at most $n\geq 1$ is nilpotent of class not exceeding a
function depending only on $n$. Using this, we can
 generalize Lemma \ref{SF} to the locally graded case, provided that the subnormal defect
 is bounded.

\begin{lem}\label{n}
Let $G$ be a locally graded group and suppose that, for some positive integer $n$,
every non-soluble subgroup of $G$ is
subnormal of defect at most $n$. Then $G$ is locally (soluble-by-finite).
\end{lem}

\begin{proof}
We may assume that $G$ is finitely generated. Suppose that it is not soluble-by-finite and
denote by $R$ its finite residual. As $G$ is locally graded, $R$ is a proper subgroup of $G$. Let $N$ be a normal
 subgroup of $G$ with finite index. Then every subgroup of $G/N$ is subnormal of defect $\leq n$ and
 so, by Roseblade's theorem \cite[12.2.3]{LR}, $G/N$ is nilpotent of bounded class depending on $n$.
  It follows that $G/R$ is nilpotent and $R$ is not soluble. Let $S$ be a proper subgroup of $R$
  and suppose that $S$ is not soluble. Then every subgroup of $R/S^R$ is subnormal of defect $\leq n$,
  in particular $R/S^R$ is soluble, by \cite[12.2.3]{LR}. This implies that $R'<R$. So $G/R'$ is
  finitely generated
  and abelian-by-nilpotent. We get that $G/R'$ is residually finite (see, for instance,
  \cite[Theorem 9.51]{Rob2})
  and $R'=R$, a contradiction. Hence every proper subgroup of $R$ is soluble and $R$ cannot be
   finite: otherwise,
   $G$ would be finite-by-nilpotent and, consequently, also nilpotent-by-finite. By Lemma \ref{hyp},
   we obtain that
   $G$ is hyperabelian. Then $G$ has a finite non-nilpotent image $G/M$ (see, for instance,
    \cite[Theorem 10.51]{Rob2}). If $M$ is soluble then $G$
is soluble-by-finite, so $M$ is not soluble and every subgroup of
$G/M$ is subnormal of defect $\leq n$. Thus $G/M$ is nilpotent by
    \cite[12.2.3]{LR}, a contradiction.
\end{proof}

\begin{proof}[{\bf Proof of Theorem \ref{dn}}]
By Lemma \ref{n}, jointly with Theorem \ref{d}, we have that $G$ is either soluble, or $G^{(r)}$ is finite
 for some $r\geq 0$ and $G$ is an extension of a soluble group of derived length at most $d$ by a finite
  almost minimal simple group. Let $G$ be soluble and denote by $e$ its derived length. We may
  assume $d<e$. Then $H=G^{(e-(d+1))}$ is soluble of length $d+1$ and every subgroup of $G$
  containing $H$ is subnormal of defect $\leq n$. It follows that $G^{(s)}\leq H$ for some
  $s$ depending on $n$ (see Remark \ref{rdependn}). Thus, $G$ is soluble of length at most $s+d+1$.
  Suppose now that there exists $r\geq 0$ such that $K=G^{(r)}$ is finite and non-soluble.
   Since every subgroup of $G$ containing $K$ is subnormal of defect $\leq n$, we get,
   as before, $G^{(t)}\leq K$ for some $t$ depending on $n$.
\end{proof}

As a final remark we point out that, in $(ii)$ of Theorems \ref{d} and \ref{dn},
one cannot expect that $G$ is an extension of a soluble group $S$ by a finite minimal
simple group: it suffices to consider the direct product of any abelian group by the
symmetric group of degree 5. However, if $M/S$ is a finite minimal simple subgroup of
 $G/S$ such that $G/S\lesssim Aut\,(M/S)$, then $M\lhd G$ by Lemma \ref{G^r} and we can
 compute the order of $G/M$. In fact, $G/M\lesssim Out\,(M/S)$ where $|Out\,(M/S)|$
 divides $2p$, with $p$ odd prime, by Table \ref{Table}.

\vspace{1cm}

\noindent{\bf Acknowledgements.} The authors would like to thank Prof. Howard Smith
for interesting discussions and useful suggestions.

\end{document}